\documentclass{amsart}

\usepackage{amssymb}
\usepackage{amsmath,amsfonts,amsthm,amstext}
\usepackage{amsmath}
\usepackage{mathrsfs}  
\usepackage[all]{xy}
\usepackage{xcolor}

\newtheorem{theorem}{Theorem}[section]

\newtheorem{lemma}[theorem]{Lemma}

\newcommand{\onto}{\twoheadrightarrow}

\newcommand{\real}{\mathbb{R}}

\newcommand{\lgoth}{\mathcal{L}}
\newcommand{\fgoth}{\mathcal{F}}

\newcommand{\lf}{\lgoth_{\fgoth}}
\newcommand{\tq}{\,|\,}
\numberwithin{equation}{section}

\def\mapnew#1{\smash{\mathop{\longrightarrow}\limits^{#1}}}

\numberwithin{equation}{section}
\def\-{\overline}

\newcommand{\z}{\mathbb{Z}}

\DeclareMathOperator{\Ker}{Ker}

\theoremstyle{definition}

\begin{document}
\title[On the $\Sigma^1$-invariant of even Artin groups ]{
On the Bieri-Neumann-Strebel-Renz $\Sigma^1$-invariant of even Artin groups} 
\author{Dessislava H. Kochloukova}
\address{State University of Campinas (UNICAMP), SP, Brazil \\
\newline
email : desi@unicamp.br
\\} 
\email{}
\date{}

\begin{abstract}  We calculate the Bieri-Neumann-Strebel-Renz invariant $\Sigma^1(G)$ for even Artin groups $G$ with underlying graph $\Gamma$ such that  if there is a closed reduced path in $\Gamma$ with all labels bigger than 2 then the length of such path is always odd. We show that $\Sigma^1(G)^c$ is a rationally defined spherical polyhedron.
\end{abstract}


\thanks{The author was   partially supported by CNPq grant 301779/2017-1  and by FAPESP grant 2018/23690-6.}

\maketitle

\section{Introduction}

By definition the $\Sigma$-invariants $\Sigma^m(G)$  and $\Sigma^m(G, \mathbb{Z})$ are subsets of the character sphere $S(G)$, where $S(G)$ is the set of equivalence classes of $  Hom(G, \mathbb{R}) \setminus \{ 0 \}$, where two real characters are equivalent if one is obtained from the other by multiplication with a positive real number. $S(G)$ can be identified with the unit sphere $\mathbb{S}^{ n-1}$ in $\mathbb{R}^n$, where $n$ is the torsion-free rank of $G/ G'$.  Whenever the invariants are defined we have the inclusions
$$S(G) = \Sigma^0(G) \supseteq \Sigma^1(G)   \supseteq \ldots \supseteq \Sigma^m(G) \supseteq \Sigma^{ m+1}(G) \supseteq \ldots$$
and
$$S(G) = \Sigma^0(G, \mathbb{Z})  \supseteq \Sigma^1(G, \mathbb{Z})  \supseteq \ldots \supseteq \Sigma^m(G, \mathbb{Z}) \supseteq \Sigma^{ m+1}(G, \mathbb{Z}) \supseteq \ldots,$$
where $\Sigma^m(G)$ is defined only for groups $G$ of homotopical type $F_m$ and $\Sigma^m (G, \mathbb{Z})$ is defined for any finitely generated group $G$ but if $G$ is not of homological type $FP_m$ then $\Sigma^m(G, \mathbb{Z}) = \emptyset$.
The homotopical finiteness property $F_m$ was first defined by Wall in \cite{W} and its homological version $FP_m$ by Bieri in \cite{B-book}. 
 By definition a group $G$ is of type $FP_m$ if the trivial $\mathbb{Z} G$-module $\mathbb{Z}$ has a projective resolution with finitely generated modules up to dimension $m$. The homotopical version of the property $FP_m$  is called $F_m$ and by definition a group $G$ is of homotopical type $F_m$ if there is a $K(G,1)$-complex with finite $m$-skeleton. Alternatively for $m \geq 2$ a group is of type $F_m$ if it is $FP_m$ and finitely presented. In \cite{BsBd}  Bestvina and Brady constructed the first examples of groups that are $FP_2$ but are not $F_2$ (i.e. finitely presented).  Furthermore these examples include groups that are of homological type $FP_{\infty}$, i.e. are $FP_m$ for every $m$,  but are not finitely presented.

 The first $\Sigma$-invariant  was defined  by 
Bieri and Strebel in \cite{BiSt}  for the class of finitely generated metabelian groups
and later the definiton was extended to homological and homotopical versions  by  Bieri, Neumann, Strebel \cite{BiNeSt}, Bieri and Renz \cite{BiRe}, Renz  \cite{Re}. In the case of a 3-manifold group the $\Sigma$-invariant defined in \cite{BiNeSt} is related to the Thurston semi-norm from \cite{Th}. Bieri and Renz showed in \cite{BiRe} that the homological $\Sigma$-invariants control the homological finiteness properties of the subgroups of $G$ that contain the commutator. The same holds in homotopical setting.

The $\Sigma$-invariants have been calculated for some classes of groups  and sometimes in low dimensions only : Thompson group $F$ \cite{BGK}, generalised Thompson groups \cite{K}, \cite{Z}, braided Thompson groups \cite{Z2}, Houghton groups \cite{Z}, some free-by-cyclic groups \cite{F-K},\cite{Ki}, fundamental groups of compact K\"ahler manifolds \cite{D}, metabelian groups of finite Pr\"ufer  rank \cite{Mein}, some Artin groups including right angled Artin groups \cite{A-K1}, \cite{A-L}, \cite{Me}, \cite{MeMnWy}, \cite{MeMnWy2}, limit groups \cite{K2}, pure symmetric automorphism groups of
finitely generated free groups \cite{O}, \cite{Z4}, some wreath products \cite{Men}, some residually free groups \cite{K-L}, the Lodha-Moore groups \cite{L-Z}. In this paper we add to the list above a subclass of the class of even Artin groups.

Let $\Gamma$ be a finite simplicial graph  with edges labeled by integers greater than one. The Artin group $G_{\Gamma}$ associated to the graph $\Gamma$ has generators the vertices $V(\Gamma)$  of $\Gamma$ and relations given by
    $$\underbrace{uvu\cdots}_{n \text{ factors}}=\underbrace{vuv\cdots}_{n \text{ factors}} \hbox{ for every edge } \hbox{ of }\Gamma \hbox{ with vertices } u,v \hbox{ and label }n \geq 2.$$
A subgraph $\Gamma_0$ of $\Gamma$ is \textbf{dominant} if, for each $v\in V(\Gamma) \setminus V(\Gamma_0)$, there is an edge $e\in E(\Gamma)$ that links $v$ with a vertex from $V(\Gamma_0)$.  

Let $\chi : G_{\Gamma} \to \real $ be a non-zero character of an Artin group $G_{\Gamma}$. An edge $e \in E(\Gamma)$ is called \textbf{dead} if $\chi(\sigma (e))=-\chi(\tau (e))$ and $e$ has an even label greater than 2.
 Let $\mathcal{L}_{\mathcal{F}}=\lf(\chi)$ be the complete subgraph of $\Gamma$ generated by the vertices $v\in V(\Gamma)$ such that $\chi(v)\neq 0$. Removing from  $\mathcal{L}_{\mathcal{F}}$ all dead edges we obtain a new graph denoted by $\mathcal{L}=\mathcal{L}(\chi)\subset \mathcal{L}_{\mathcal{F}}$ and called the \textbf{living subgraph}. Inspired by the Meier results in \cite{Me}  Almeida  and Kochloukova suggested in \cite{A-K1} the following conjecture.
 
 \medskip
 {\bf The $\Sigma^1$-Conjecture for Artin groups}
 {\it Let $G = G_{\Gamma}$ be an Artin group. Then
  $$\Sigma^1(G)=\{[\chi]\in S(G)\tq \mathcal{L}(\chi) \text{ is  a connected dominant subgraph  of $\Gamma$} \}.$$}
  
The $\Sigma^1$-Conjecture for Artin groups is known to hold when $\Gamma$ is a tree \cite{Me}, $\pi_1(\Gamma) \simeq \mathbb{Z}$ \cite{A-K1}, $\pi_1(\Gamma) $ is a free group of rank 2 \cite{A}, for Artin groups of finite type (i.e. the associated Coheter group is finite) \cite{A-L}. We concentrate now on a special type of Artin groups : the even Artin groups i.e. all the labels of the graph $\Gamma$ are even numbers. Though there is an extensive literature on Artin groups only recently even Artin groups were considered with more details.  In  \cite{G2} Blasco-Garcia,  Martinez-Perez and Paris showed that the even Artin groups of FC type are poly-free and residually finite. The $\Sigma$-invariants of even Artin groups of FC type are studied in \cite{C}. 
 Our main result describes $\Sigma^1(G)$ for specific even Artin groups $G$ that are not required to be of FC type.

\medskip
{\bf Theorem A}
  {\it Let $G = G_{\Gamma}$ be an even Artin group such that if there is a closed reduced path in $\Gamma$ with all labels bigger than 2 then the length of such path is always odd. Then
  $$\Sigma^1(G)=\{[\chi]\in S(G)\tq \mathcal{L}(\chi) \text{ is a connected dominant subgraph  of $\Gamma$} \}.$$}

Our approach to prove Theorem A is homological in the sense that we obtain information about $G'/ G''$ by calculating it as a certain homology group. In the calculation part we use Fox derivatives in order to construct the beginning of a free resolution of the trivial $\mathbb{Z} G$-module $\mathbb{Z}$.

In \cite{B-Gr} using valuation theory Bieri and Groves  showed  that for a finitely generated metabelian group $G$ the complement $\Sigma^1(G)^c = S(G) \setminus \Sigma^1(G)$ is a rationally defined spherical polyhedron i.e. a finite union of finite intersections of closed rationally defined semi-spheres, where rationality means that each semisphere is defined by a vector in $\mathbb{R}^n$ with rational coordinates. 
In \cite{Ki} Kielak considered a non-normalised version of $\Sigma^1(G)$, where $\Sigma^1(G)$ is a subset of $Hom(G, \mathbb{R}) \setminus \{ 0 \}$ and he showed that for one of the classes of groups :
descending HNN extensions of finitely generated free groups; 
fundamental groups of compact connected oriented 3-manifolds; 
agrarian Poincaré duality groups of dimension 3 and type F; 
agrarian groups of deficiency one $\Sigma^1(G)$ is determined by an integral polytope \cite{Ki}. Using the normalised definition of $\Sigma$ in many cases $\Sigma^1(G)^c$ is a rationally defined  spherical polyhedron but this is not true in general. In \cite{BiNeSt} there is an example of a group of PL automorphisms of an interval where $\Sigma^1(G)^c$ contains a non-discrete isolated point, hence it is not a rationally defined spherical polyhedron.

For a subgroup $H$ of $G$ we set
$$S(G,H):=\{[\chi]\in S(G)~ \tq ~ \chi(H)=0\}.$$
Note that $S(G, H)$ is a rationally defined spherical polyhedron.

\medskip
{\bf Proposition B} {\it Let $G = G_{\Gamma}$ be an Artin group for which the $\Sigma^1$-Conjecture for Artin groups holds. Then 
there exist subgroups $H_1 \ldots, H_s$ of $G$ such that
$$\Sigma^1(G)^c = \cup_{1 \leq i \leq s} S(G, H_i),$$
in particular
$\Sigma^1(G)^c$ is a rationally defined spherical  polyhedron. }

\medskip

As a corollary of the proof of Theorem A we obtain the following result.

\medskip
{\bf Corollary C} {\it Let $G = G_{\Gamma}$ be an even Artin group such that if there is a closed reduced path in $\Gamma$ with all labels bigger than 2 then the length of such path is always odd. Then for  $[\chi] \in S(G)$ we have  $[\chi] \in \Sigma^1(G)$ if and only if $ [\overline{\chi}] \in \Sigma^1(G/ G '')$, where $\overline{\chi}$ is the real character of $G/G''$ induced by $\chi$.}

\medskip
In Section \ref{final} we explain why the approach we use in the proof of Theorem A  does not work for more general Artin groups. The point is that for more general Artin groups (including some even Artin groups) Corollary C does not necessary hold. We provided concrete examples (one is an even Artin group and the other is not) where the expected behaviour of $\Sigma^1(G)$, as suggested by the $\Sigma^1$-Conjecture for Artin groups, differs from the structure of $\Sigma^1(G/ G'')$.

\section{Preliminary results} \label{prelim}

\subsection{On the $\Sigma$-invariants} 

 Let $G$ be a finitely generated group.
By definition $S(G)$ is the set of equivalence classes $[\chi]$ of non-zero real characters $$\chi : G \to \mathbb{R}$$ with respect to the equivalence relation $\sim$, given by: $\chi_1 \sim \chi_2$ if and only if there is a positive real number $r$ such that $\chi_1 = r \chi_2$. Thus $[\chi] = R_{>0} \chi$. For a fixed character $\chi: G \to \mathbb{R}$ consider the monoid  $$G_{\chi} = \{ g \in G  ~| ~ \chi(g) \geq 0 \}.$$ Bieri and Renz defined in \cite{BiRe} the homological $\Sigma$-invariants as
$$
\Sigma^m(G, A) = \{ [\chi] \in S(G) \tq A \hbox{ has type } FP_m \hbox{ as } \z[G_{\chi}]\hbox{-module} \},
$$ 
where $A$ is a finitely generated $\z G$-module. 
In this paper all considered modules are left ones. 

There is a homotopical version of $\Sigma^m(G, \mathbb{Z})$ where $\z $ is the trivial $\z G$-module, denoted by  $\Sigma^m(G)$. By definition $[\chi] \in \Sigma^1(G)$ if and only if for a Cayley graph $\Gamma$ of $G$, that corresponds to a finite generating set of $G$, the subgraph $\Gamma_{\chi}$ spanned by all vertices in $G_{\chi}$ is connected. In some papers the invariants are defined by left actions and in others by right actions, thus they might differ by sign but if all actions come from the same side  the known definitions of $\Sigma$-invariants agree.
By \cite{Re} for any group $G$ we have  $\Sigma^1(G) = \Sigma^1(G, \z)$   and for $m \geq 2$, $\Sigma^m(G) = \Sigma^m(G, \mathbb{Z}) \cap \Sigma^2(G)$. Furthermore by \cite{MeMnWy2} the inclusion $\Sigma^m(G) \subseteq \Sigma^m(G, \mathbb{Z})$ could be strict.

\begin{theorem} \cite{BiRe} \label{teo sigma prop fin}
    Let $G$ be a group of type $FP_m$  and $H$ be a subgroup of $G$ containing  the commutator. Then $H$ is of type $FP_m$ if and only if
    $$S(G,H) =\{[\chi]\in S(G)~ \tq ~ \chi(H)=0\}\subseteq \Sigma ^m(G,\mathbb{Z}).$$
\end{theorem}

\begin{theorem} \cite{BiRe} \label{teosigmaopen}
 Let $G$ be a finitely generated group and $A$ a finitely generated $\z G$-module. Then $\Sigma^m(G,A)$ is open in $S(G)$. If $G$ is of type $F_m$ then $\Sigma^m(G)$ is open in $S(G)$.
\end{theorem}

For a subset $T \subseteq S(G)$ we write $T^c$ for the complement $S(G) \setminus T$ of $T$ in $S(G)$.
The following result is a well-known  fact that follows from the definition of $\Sigma^1$.

\begin{lemma}\label{propsigmaepi}
Let $\pi:G\onto H$ be a group epimorphism and ${\mu}$ be a non-trivial character of $H$. Then
  $\left[{\mu}\right]\in\Sigma ^1\left(H\right)^c$ implies $ \left[{\mu}\circ\pi\right]\in\Sigma ^1(G)^c.$
\end{lemma}

\subsection{$\Sigma$-invariants of Artin groups}

Consider the living subgraph $\mathcal{L}=\mathcal{L}(\chi)$ that was defined in the introduction.

\begin{theorem} \cite{MeMnWy} \label{teobase}
  Let $G = G_{\Gamma}$ be an Artin group  and let $\chi$ be a non-zero real character of $G$.

  \begin{enumerate}
    \item
    If $\mathcal{L}(\chi)$ is a connected dominant subgraph of $\Gamma$ then $[\chi]\in\Sigma^1(G)$.

    \item
    If $[\chi]\in \Sigma^1(G)$ then $\mathcal{L}_{\mathcal{F}}(\chi)$ is a connected dominant subgraph of $\Gamma$.
  \end{enumerate}
\end{theorem}

 Using Theorem \ref{teobase} the $\Sigma^1$-Conjecture for Artin groups was reduced in \cite{A-K1}  to the case of discrete characters.

\begin{lemma}\label{corprob} \cite{A-K1}
  Let $G = G_{\Gamma}$ be an Artin group. Assume that for every discrete character $\chi$ of $G$ such that
  $\mathcal{L}_{\mathcal{F}}(\chi)$  is connected and $ \mathcal{L}(\chi) $ is  disconnected  then $ [\chi]\in\Sigma^1(G)^c. $
  Then
  $$\Sigma^1(G)=\{[\chi]\in S(G)\tq \mathcal{L}(\chi) \text{ is a connected dominant subgraph of $\Gamma$} \}.$$
\end{lemma}

We state some easy facts that were observed in \cite{A-K1}.

\begin{lemma}\label{lema simetria sigma1} \cite{A-K1}
  Let $G_{\Gamma}$ be an Artin group. 
  
  a) $\Sigma ^1(G_{\Gamma})=-\Sigma ^1(G_{\Gamma})$;
  
  b) if $\chi$ a discrete character of $G$ then
  $[\chi]\in\Sigma^1(G_{\Gamma}) $ if and only if $ \Ker(\chi) $  is finitely generated.
\end{lemma}

\section{Fox derivatives} \label{Fox-section}

Let $G$ be a group and $M$ be a left $\mathbb{Z} G$-module. A function
$$d : G \to M$$ is called a derivation if
$$d(gh) = d(g) + g d(h) \hbox{ for every } g,h \in G.$$
Let $F = F(X)$ be a free group with a basis $X$, $x \in X$ and consider the derivation
$$
\frac{\partial}{ \partial x } : F \to \mathbb{Z} F$$
given by $\frac{\partial }{\partial x }(y)= \delta_{x,y}$ the Kroniker symbol for every $y \in X$.

Suppose now that $G$ has a presentation $$\langle X | r_1, \ldots, r_s \rangle.$$  Then by \cite[Prop. 5.4, Chapter II + Exer.~3,p. 45]{Brown} there is an exact complex of left $\mathbb{Z} G$-modules
\begin{equation}  \label{exact1} 
\oplus_{1 \leq i \leq s} \mathbb{Z} G e_{r_i} ~\mapnew{\partial_2}   ~\oplus_{x \in X} \mathbb{Z} G e_x ~\mapnew{\partial_1} ~\mathbb{Z} G ~\mapnew{\partial_0} ~ \mathbb{Z} ~\mapnew{} ~ 0, \end{equation} 
where $\partial_0 = \epsilon$ is the augmentation map, $\partial_1(e_x) = \overline{x} - 1$ where overlining denotes the image of an element of $\mathbb{Z} F$ in $\mathbb{Z} G$ and for $r \in \{ r_1, \ldots, r_s \}$
$$\partial_2( e_r) = \sum_{x \in X} \overline{\frac{\partial }{\partial x}(r)} e_x.$$
We fix the commutator notation $[a,b] = a^{ -1} b^{ -1}  ab$.

\begin{lemma} \label{Fox-1} Let $s \in X$ and $a,b \in F(X)$. Then \begin{equation} \label{Fox}  \frac{\partial}{\partial s} ( [a,b]) =a^{ -1} ( b^{ -1} - 1) \frac{\partial}{\partial s}  ( a ) + a^{ -1}  b^{  - 1} (a -1) \frac{\partial}{\partial s}  ( b).\end{equation} 
In particular for $x, y \in X$ with $x \not= y$ and $r = [x,y] \in \{ r_1, \ldots, r_s \}$ we have
$$\partial_2(e_r) = \overline{ x^{-1} y^{-1}} (  \overline{ (x-1)} e_y - \overline{ (y - 1)} e_x).$$
\end{lemma} \begin{proof}
$$\frac{\partial}{\partial s} ( [a,b]) = \frac{\partial}{\partial s}  ( a^{ -1} b^{ -1} a b) = \frac{\partial}{\partial s}  ( a^{ -1} )  + a^{ -1}  \frac{\partial}{\partial s}  ( b^{-1}) + a^{ -1} b^{  - 1} \frac{\partial}{\partial s} ( a) + a^{ -1} b^{  - 1} a \frac{\partial}{\partial s}  ( b)  =$$
$$- a^{ -1}\frac{\partial}{\partial s}  ( a )  - a^{ -1} b^{ -1} \frac{\partial}{\partial s}  ( b) + a^{ -1} b^{  - 1} \frac{\partial}{\partial s} ( a) + a^{ -1} b^{  - 1} a \frac{\partial}{\partial s}  ( b)=$$
$$
a^{ -1} ( b^{ -1} - 1) \frac{\partial}{\partial s}  ( a ) + a^{ -1}  b^{  - 1} (a -1) \frac{\partial}{\partial s}  ( b).$$
\end{proof}
\begin{lemma} \label{Fox-2} Suppose $x, y \in X$ with $x \not= y$ and $m \geq 2$ an integer. Then for  $r = [(xy)^m, x]$ we have
$$\frac{\partial }{ \partial y}(r) = (xy)^{ -m} ( x^{ -1} - 1)( 1 + xy + \ldots + (xy)^{ m-1})  x$$ and $$\frac{\partial }{ \partial x} (r) = 
(xy)^{ -m}( y - 1) ( 1 + xy + \ldots + (xy)^{ m-1})
.$$
In particular if $r \in \{ r_1, \ldots, r_s \}$  we have that the image $\lambda$ of $\partial_2(e_r)$ in $
\mathbb{Z} Q e_x \oplus \mathbb{Z} Q e_y$  is
$$\lambda = (x_0y_0)^{ -m} ( 1 + x_0y_0 + \ldots + (x_0y_0)^{ m-1}) ((y_0 -1) e_x - (x_0-1) e_y),
$$
where  $Q = G/ G'$, $x_0$ is the image of $x$ in $Q$ and $y_0$ is the image of $y$ in $Q$.
\end{lemma}

\begin{proof}
By (\ref{Fox})
$$\frac{\partial}{\partial s}  (r) = \frac{\partial}{\partial s} ( [(xy)^m, x]) =$$ $$ (xy)^{ -m} ( x^{ -1} - 1) \frac{\partial}{\partial s}  ( (xy)^m ) + (xy)^{ -m}  x^{  - 1} ( (xy)^m -1) \frac{\partial}{\partial s}  ( x) = $$
$$ (xy)^{ -m} ( x^{ -1} - 1)( 1 + xy + \ldots + (xy)^{ m-1}) \frac{\partial}{\partial s}  ( xy) + (xy)^{ -m}  x^{  - 1} ( (xy)^m -1) \frac{\partial}{\partial s}  ( x).$$
Then
$${\frac{\partial}{\partial x}  (r)} = (xy)^{ -m} ( x^{ -1} - 1)( 1 + xy + \ldots + (xy)^{ m-1}) \frac{\partial}{\partial x}  ( xy) +$$ $$ (xy)^{ -m}  x^{  - 1} ( (xy)^m -1) \frac{\partial}{\partial x}  ( x)=$$
$$
(xy)^{ -m} ( x^{ -1} - 1)( 1 + xy + \ldots + (xy)^{ m-1}) + (xy)^{ -m}  x^{  - 1} ( (xy)^m -1) =$$
$$
(xy)^{ -m}( x^{ -1} - 1 +  x^{  - 1}(xy - 1)) ( 1 + xy + \ldots + (xy)^{ m-1})  = 
$$
$$
(xy)^{ -m}( y - 1) ( 1 + xy + \ldots + (xy)^{ m-1})
$$
and
$$\frac{\partial}{\partial y}  (r) = (xy)^{ -m} ( x^{ -1} - 1)( 1 + xy + \ldots + (xy)^{ m-1}) \frac{\partial}{\partial y}  ( xy) +$$ $$ (xy)^{ -m}  x^{  - 1} ( (xy)^m -1) \frac{\partial}{\partial y}  ( x)= (xy)^{ -m} ( x^{ -1} - 1)( 1 + xy + \ldots + (xy)^{ m-1})  x.$$
 Then for the image $\lambda$ of $\partial_2(e_r)$ in $
\mathbb{Z} Q e_x \oplus \mathbb{Z} Q e_y$  we have
$$\lambda = (x_0y_0)^{ -m} ( y_0 - 1)( 1 + x_0y_0 + \ldots + (x_0y_0)^{ m-1})  e_x + $$ $$(x_0y_0)^{ -m} ( x_0^{ -1} - 1)( 1 + x_0y_0 + \ldots + (x_0y_0)^{ m-1})  x_0 e_y = $$
$$ (x_0y_0)^{ -m} ( 1 + x_0y_0 + \ldots + (x_0y_0)^{ m-1}) \gamma,$$
where
$\gamma = (y_0 - 1) e_x + (x_0^{ -1 } -1)x_0 e_y = (y_0 -1) e_x - (x_0-1) e_y.$

\end{proof}
\section{One algebraic lemma} 

\begin{lemma} \label{alg} Let $T$ be a forest,  where the set of vertices is a disjoint union $V \cup W$ and each edge links a vertex of $V$ with a vertex of $W$ that is labelled by an even positive integer $2 m_{v,w} > 2$. Let $Q$ be the free abelian group with free basis $V \cup W$ and consider the $\mathbb{Z} Q$-module
$$K_{T} = (\oplus_{v \in V, w \in W} \mathbb{Z} Q  f_{v,w} )/ J,$$
where $\mathbb{Z} Q  f_{v,w}$ is a free $\mathbb{Z} Q$-module and $J$ is the $\mathbb{Z} Q$-submodule generated by 
$$\{ (1 + vw + \ldots + (v w)^{ m_{v,w} - 1}) f_{v,w} ~| ~v \in V, w \in W, v,w\hbox{ are vertices of an edge in } T\},$$
 $$\{ (t-1) f_{v,w} - (v-1) f_{t,w} ~| ~ v,t \in V, w \in W \}$$ and
$$\{ (s-1) f_{v,w} - (w-1) f_{v,s} ~| ~ v \in V, w,s \in W \}. $$
Let $$\chi : Q \to \mathbb{Z}$$ be a character (i.e.  homomorphism of groups) such that for every vertex $u \in V(\Gamma)$ we have $\chi(u) \not= 0$ and for every edge in $T$ with vertices $v$ and $w$ we have $\chi(v) = - \chi(w)$.
Then $K_{T}$ is not finitely generated as $\mathbb{Z} \Ker (\chi)$-module.
\end{lemma}

\begin{proof} 

Suppose that  there is an edge in $T$ with vertices $v \in V, w \in W$. We define $\lambda_{v,w}  \in \mathbb{C}$ as a primitive $m_{v,w}$-th root of 1. 

We decompose $T $ as a disjoint union $\cup_i T_i$ where each $T_i$ is a tree.
Let $v_i$ be a  fixed base point of $T_i$.  
Consider the Laurent polynomial ring $\mathbb{C} [x^{ \pm 1}]$ and the ring homomorphism
$$\mu : \mathbb{C} Q \to \mathbb{C} [x^{  \pm 1}]$$
that sends 

1) $vw$ to $\lambda_{v,w}$ for all $v \in V, w \in W$ vertices of an edge of $T$, 

2) $v_i$ to $x^{ \chi(v_i)}$,

3) is the identity on $\mathbb{C}$.

Note that $\mu$ is well-defined because $T$ is a forest and that $$\mu(\Ker(\chi)) \subseteq \mathbb{C}.$$ Furthermore for $q \in Q$ we have that
$$\mu(q) \in \mathbb{C} x^{ \chi(q)},\hbox{ hence }
\mu(Q) \subseteq (\cup_{ j \in \mathbb{Z} \setminus \{ 0 \} }\mathbb{C} x^j).$$
Then $\mu$ induces an epimorphism of $\mathbb{C} Q$-modules $$\widehat{\mu} : \mathbb{C} \otimes_{\mathbb{Z}}  K_{T}  \to (\oplus_{v \in V, w \in W}  \mathbb{C} [ x^{  \pm 1}] f_{v,w})/ R = : \widehat{K}_{T}$$
that is the identity on $f_{v,w}$,
where $R$ is the $\mathbb{C} [ x^{  \pm 1}]$-submodule generated by
 $$\{ (\mu(t)-1) f_{v,w} - (\mu(v)-1) f_{t,w} ~| ~ v,t \in V, w \in W \}$$
 and
$$\{ (\mu(s)-1) f_{v,w} - (\mu(w)-1) f_{v,s} ~| ~ 
v \in V, w,s \in W\}$$ 
and we view $\widehat{K}_{T}$ as $\mathbb{C} Q$-module via $\mu$.

Let $D$ be the localised ring $\mathbb{C}[x^{  \pm 1}, \{ \frac{1}{c x^j- 1} ~| ~j \in \mathbb{Z} \setminus \{ 0 \},  c \in \mathbb{C} \} ]$. Since the coefficients of $\mathbb{C} Q$ that appear in the generators of $R$ are invertible in $D$, we  have that $D \otimes_{\mathbb{C}[v_0^{  \pm 1} ]} \widehat{K}_T$ is a cyclic $D$-module isomorphic to $D$.  Hence for the annihilator of $D \otimes_{\mathbb{C}[x^{  \pm 1} ]} \widehat{K}_T$ in $D$ we have
$ann_D( D \otimes_{\mathbb{C}[x^{  \pm 1}]} \widehat{K}_T) = ann_D(D) = 0$, thus
$ann_{ \mathbb{C}[x^{  \pm 1} ] } (\widehat{K}_T) = 0$.
This implies that
$$\widehat{K}_{T} \hbox{ is infinite dimensional over } \mathbb{C}.$$

Finally if $K_{T}$ is finitely generated as $\mathbb{Z} \Ker(\chi)$-module then $ \mathbb{C}  \otimes_{\mathbb{Z}} K_{T}$ is finitely generated as $\mathbb{C} \Ker(\chi)$-module. Then applying $\widehat{\mu}$  we obtain that
$\widehat{K}_{T}$ is finitely generated as ${\mu}(\mathbb{C} \Ker(\chi))$-module. But since $\mu( \Ker(\chi)) \subseteq \mathbb{C}$ we deduce that ${\mu}(\mathbb{C} \Ker(\chi)) = \mathbb{C} \mu(\Ker(\chi)) = \mathbb{C}.$ Thus $\widehat{K}_{T}$ is finite dimensional over $\mathbb{C}$, a contradiction.
\end{proof}

\section{Proof of Theorem A} 
 Let $G = G_{\Gamma}$ be an even Artin group such that if there is a closed reduced path in $\Gamma$ with all labels bigger than 2 then the length of such path is always odd. Let $\chi : G \to \mathbb{Z}$ be a non-zero character such that  $\mathcal{L}_{\mathcal F}(\chi)$ is connected and $\mathcal{L}(\chi)$ is not connected. 
By Lemma \ref{corprob} to prove Theorem A is equivalent to show that $[\chi] \notin \Sigma^1(G)$.

\medskip
I. We suppose first that $\chi(v ) \not= 0$ for every $v \in V(\Gamma)$. 
Since the graph $\mathcal{L}(\chi)$  is not connected $V(\Gamma) = V(\Gamma_1) \cup V(\Gamma_2),$ where the union is disjoint, $\Gamma_1$ and $\Gamma_2$ are complete subgraphs of $\Gamma$ ( i.e. if an edge $e$ connects two vertices from $\Gamma_i$ in $\Gamma$ then this edge belongs to $E(\Gamma_i)$) and there is  a disjoint union $E(\Gamma) = E(\Gamma_1) \cup E(\Gamma_2)  \cup E,$ where $E$ is a set of dead edges  (with respect to the character $\chi$) i.e. if $u,v$ are the vertices of such an edge then $\chi(u) = - \chi(v)$. Let $\gamma$ be a closed reduced path with edges in $E$. The bipartite structure of the graph $\Gamma_0$ implies that $\gamma$ has  even length, a contradiction with the assumption that a closed reduced path in $\Gamma$ with all labels bigger than 2 has always odd length. Hence no such path $\gamma$ exists i.e. $E$ is a forest.

Let ${\Gamma}_0$ be the graph with a set of vertices $V({\Gamma}_0) = V(\Gamma)$, for $ i = 1,2 $ every two vertices of $V(\Gamma_i)$ are linked in $\Gamma_0$ by an edge with label 2  and the edges in $\Gamma_0$  that link some vertices of $V(\Gamma_1)$ with some  vertices of $V(\Gamma_2)$ are precisely the edges from $E$  with their original labels in $\Gamma$. 
Let $G_0 = G_{\Gamma_0}$ be the Artin group with underlying graph $\Gamma_0$.
Our aim is to calculate $G_0'/ G_0''$ via the isomorphism $$G_0'/ G_0'' \simeq H_1(G_0', \mathbb{Z}).$$ By (\ref{exact1}) there is an exact complex 
$$
{\mathcal P} : \oplus_r  ~ \mathbb{Z}G_0 e_r  \mapnew{\partial_2} \oplus_v \mathbb{Z} G_0  e_v \mapnew{\partial_1} \mathbb{Z} G_0 \to \mathbb{Z}\to 0,$$
where each relation $r$ corresponds to one edge of $\Gamma_0$ and $v$ runs through the vertices $V(\Gamma_0)$.
We have two type of relations : $r = [v_1, v_2]$ where both $v_1, v_2 $ are vertices of $V(\Gamma_i)$ for a fixed $i \in \{ 1,2 \}$ or $r = [(vw)^m, v]$, where  $ v \in V(\Gamma_1)$, $ w \in V(\Gamma_2)$ and there is an edge in $E$ with label $2m> 2$ that links $v$ and $w$. By (\ref{exact1})
$$\partial_1(e_v) = \overline{ v}-1 \in \mathbb{Z} G_0  ~\hbox{ 
and } ~
\partial_2( e_r) = \sum_{v} \overline{\frac{\partial }{ \partial v} (r)} e_v.$$
 Thus for $Q = G_0/ G_0'$ free abelian group with a free abelian basis $X = V(\Gamma_0)$ we have the complex
$$
\mathcal {S} = \mathbb{Z} \otimes_{\mathbb{Z} G_0'} \mathcal{P} :  ~ \oplus_r \mathbb{Z}Q  e_r  ~\mapnew{\widetilde{\partial}_2} ~\oplus_v  \mathbb{Z}Q  e_v  ~\mapnew{\widetilde{\partial}_1}~\mathbb{Z} Q ~\mapnew{} ~ \mathbb{Z}~\mapnew{}  ~0$$ and
$$H_1(\mathcal{S}, \mathbb{Z}) = Ker (\widetilde{\partial}_1)/ Im (\widetilde{\partial}_2).$$
By abuse of notation we write $v$ for the image of $v \in V(\Gamma_0) \subset G_0$ in $Q = G_0/ G_0'$. Thus
$$\widetilde{\partial}_1(e_v) = v-1.$$By Lemma 
\ref{Fox-1}   for $r = [v_1, v_2]$ for $v_1, v_2 \in V(\Gamma_i)$, $v_1 \not= v_2$, $i = 1,2$ we have 
$$\widetilde{\partial}_2(e_r) = v_1^{-1} v_2^{-1}(  ( v_1-1) e_{v_2} - (v_2 - 1)e_{v_1})$$ and by
Lemma \ref{Fox-2} for $r = [(v_1v_2)^m, v_2]$ for some $v_1, v_2 \in V(\Gamma)$, where $v_1$ and $v_2$ are linked by an edge from $E$ with label $2m$, we have 
$$\widetilde{\partial}_2(e_r) = (v_1 v_2)^{ -m} ( 1 + v_1 v_2+ \ldots + (v_1v_2)^{ m-1}) ((v_2 -1) e_{v_1} - (v_1-1) e_{v_2}).$$
On the other hand we have the exact Koszul complex
$$
{\mathcal R} : \ldots  \to \oplus_{v_1 < v_2 < v_3} \mathbb{Z} Q e_{v_1} \wedge e_{v_2} \wedge e_{v_3}  \mapnew{d_3} \oplus_{v_1 < v_2} \mathbb{Z} Q e_{v_1} \wedge e_{v_2}  \mapnew{d_2}  
$$ $$\oplus_{v_1}  \mathbb{Z} Q  e_{v_1} \mapnew{d_1} \mathbb{Z} Q \to \mathbb{Z} \to 0,$$
where in the direct sums $v_1, v_2, v_3 \in V(\Gamma_0)$, $\leq $ is a fixed linear order of $V(\Gamma)$ and ${\mathcal R}$ has differentials
$d_1 = \widetilde{\partial}_1,$
$$d_2(e_{v_1} \wedge e_{v_2}) =  (v_2 - 1) e_{v_1} - (v_1 - 1) e_{v_2} $$
and
$$d_3 (e_{v_1} \wedge e_{v_2} \wedge e_{v_3} ) = (v_1 -1) e_{v_2} \wedge e_{v_3}  + (v_2 - 1) e_{v_3} \wedge e_{v_1} +  (v_3 - 1) e_{v_1} \wedge e_{v_2}.$$ Then
$$H_1(\mathcal{S}, \mathbb{Z}) = Ker (\widetilde{\partial}_1)/ Im (\widetilde{\partial}_2) = Ker ( d_1)/ Im (\widetilde{\partial}_2) = Im ( d_2)/ Im (\widetilde{\partial}_2).$$
Note that by the exactness of the Koszul complex $\mathcal{R}$ we have
$$Im (d_2) \simeq  (\oplus_{v_1 <  v_2}  \mathbb{Z} Q e_{v_1} \wedge e_{v_2}  )/ Ker (d_2) = (\oplus_{v_1 < v_2}  \mathbb{Z} Q e_{v_1} \wedge e_{v_2} ) / Im (d_3),$$ where the first isomorphism is given by the classical theorem of the isomorphism (associated to any homomorphism of groups). This isomorphism sends $Im (\widetilde{\partial}_2)$ to $d_2^{ -1} ( Im (\widetilde{\partial}_2))/ Im (d_3)$.
Hence
$$
Im ( d_2)/ Im (\widetilde{\partial}_2) \simeq (\oplus_{v_1 < v_2}  \mathbb{Z} Qe_{v_1} \wedge e_{v_2} ) / d_2^{ -1} ( Im (\widetilde{\partial}_2)) = : M,$$
where
$d_2^{ -1} ( Im (\widetilde{\partial}_2))$ is the $\mathbb{Z} Q$-submodule of $\oplus_{v_1 < v_2}  \mathbb{Z} Qe_{v_1} \wedge e_{v_2} $ generated by  $Im (d_3)$, 
$e_{v_1} \wedge e_{v_2}$ if both $v_1, v_2$ are vertices in $\Gamma_i$ (for $i = 1,2$) and  $( 1 + vw + \ldots + (vw)^{ m-1}) e_v \wedge e_w $ if $v \in V(\Gamma_1), w \in V(\Gamma_2)$ and $v,w$ are vertices of an edge from $E$ with label $2m > 2$.
 
 Consider the projection
$$\pi : \oplus_{v_1 <  v_2}  \mathbb{Z} Q e_{v_1} \wedge e_{v_2} \to \oplus_{v \in V(\Gamma_1), w \in V(\Gamma_2)}  \mathbb{Z} Q  e_v \wedge e_w $$ that sends $e_{v_1} \wedge e_{v_2}$ to 0 if both $v_1, v_2$ are vertices in $\Gamma_i$ (for $i = 1,2$)  and is the identity on $e_v \wedge e_w$ if $v \in V(\Gamma_1), w \in V(\Gamma_2)$. Note that $\Ker (\pi) \subseteq d_2^{ -1} ( Im (\widetilde{\partial}_2))$.
Then $\pi$ induces the isomorphism
$$
M \simeq (\oplus_{v \in V(\Gamma_1), w \in V(\Gamma_2)} \mathbb{Z} Q e_v \wedge e_w  )/ N,$$
where
$N = \pi( d_2^{ -1} ( Im (\widetilde{\partial}_2)) )$ is the $\mathbb{Z} Q$-submodule generated by $$ \{ ( 1 + vw + \ldots + (vw)^{ m-1}) e_v \wedge e_w ~|$$ $$ ~v \in V(\Gamma_1),w \in V(\Gamma_2), v,w \hbox{ are vertices of } e \in E \hbox{ with label }2m > 2 \} $$ $$\cup ~ \{ \pi(d_3(e_v \wedge e_w \wedge e_t)) ~| ~ v \in V(\Gamma_1),w \in V(\Gamma_2), t \in V(\Gamma)  = V(\Gamma_1) \cup V(\Gamma_2) \}.$$
Thus if $t \in V(\Gamma_1)$ we have
$$ \pi(d_3(e_v \wedge e_w \wedge e_t)) = \pi(  (t-1) e_v \wedge e_w + (v-1) e_w \wedge e_t  + ( w - 1) e_t \wedge e_v )= $$ $$ (t-1)e_v \wedge e_w + (v-1) e_w \wedge e_t $$
and if $t \in V(\Gamma_2)$ we have
$$ \pi(d_3(e_v \wedge e_w \wedge e_t)) = \pi(  (t-1) e_v \wedge e_w + (v-1) e_w \wedge e_t  + ( w - 1) e_t \wedge e_v )= $$
$$ (t-1) e_v \wedge e_w  + ( w - 1) e_t \wedge e_v. $$
Note that using the notation $K_T$ from  Lemma \ref{alg} for $T = E$, $V = V(\Gamma_1)$ and $W = V(\Gamma_2)$ we have proved by now that
\begin{equation} \label{vazno} G_0'/ G_0'' \simeq H_1(G_0, \mathbb{Z}) \simeq (\oplus_{v \in V(\Gamma_1), w \in V(\Gamma_2)} \mathbb{Z} Q e_v \wedge e_w  )/ N \simeq  K_E.\end{equation}

Consider the epimorphism of groups
$$\theta: G = G_{\Gamma} \to G_0 = G_{\Gamma_0}$$
which is the identity on $V(\Gamma)$ and note that $Ker(\theta) \subseteq G'$. Let $\chi_0 : G_0 \to \mathbb{Z}$ be the character induced by $\chi$, i.e. $\chi_0 \circ \theta = \chi$, and
 $\nu_0 : G_0/ G_0'\to \mathbb{Z}$ and $\overline{\chi}_0 : G_0 / G_0'' \to \mathbb{Z}$ be the characters induced by $\chi_0$.
By Lemma \ref{alg} $K_E$ is not finitely generated as $\mathbb{Z} \Ker(\nu_0)$-module, hence by (\ref{vazno})
$$G_0'/ G_0'' \hbox{ is not finitely generated  as }\mathbb{Z} \Ker(\nu_0)-\hbox{module.}$$ Furthermore since $G_0$ has an automorphism $\varphi$ that sends each vertex $v \in V(\Gamma_0)$ to $v^{-1}$, we note that $\varphi$ induces an automorphism of $G_0/G_0''$. This implies that $\Sigma^1(G_0/ G_0'') = - \Sigma^1(G_0/ G_0 '')$. Hence \begin{equation} \label{eqq} [\overline{\chi}_0] \notin \Sigma^1(G_0/G_0'')\end{equation} otherwise $[\overline{\chi}_0] \in \Sigma^1(G_0/G_0'') \cap - \Sigma^1(G_0/G_0'')$, $\Ker(\overline{\chi}_0)$ is finitely generated and hence $G_0' /G_0''$ is finitely generated as $\mathbb{Z} \Ker(\nu_0)$-module, a contradiction.   

Let $\overline{\chi} : G/ G'' \to \mathbb{Z}$ be the character  induced by $\chi$. Note that since $G_0/ G_0''$ is a quotient of $G/ G''$, Lemma \ref{propsigmaepi}  and (\ref{eqq}) imply that 
\begin{equation}  \label{eqq2} [\overline{\chi}] \notin \Sigma^1(G/ G''). \end{equation}
By (\ref{eqq2}) and Lemma \ref{propsigmaepi} applied for the canonical epimorphism $G \to G/ G''$  we get that $[\chi] \notin \Sigma^1(G)$ as required.

\medskip	
	II. Now we consider the general case when $\chi$ might have 0 values on some vertices of $\Gamma$. In this case we consider $\widehat{\Gamma} = \mathcal{L}_{\mathcal{F}(\chi)}$ the graph obtained from $\Gamma$ by deleting all vertices $v$ with $\chi(v) =0$ and all edges that have a vertex $v$ as above. 
	Consider the epimorphism
	$$\pi : G = G_{\Gamma} \to \widehat{G} = G_{\widehat{\Gamma}}$$ that is identity on $v \in V(\widehat{\Gamma})$ 
	and sends $v \in V(\Gamma) \setminus V(\widehat{\Gamma})$ to $1$. Let $\widehat{\chi} : \widehat{G} \to \mathbb{Z}$ be the character induced by $\chi$ and  $\mu : \widehat{G}/ \widehat{G}'' \to \mathbb{Z}$ be the character induced by $\widehat{\chi}$. Then by case I we have 
$$[ \mu ] \notin \Sigma^1(\widehat{G}/ \widehat{G}'' ) \hbox{ and } [\widehat{\chi}] \notin \Sigma^1( \widehat{G}). $$
Since $\widehat{G}/ \widehat{G}'' $ is a quotient of $G/ G''$, for the character $\overline{\chi} : G/ G '' \to \mathbb{Z}$ induced by $\chi$ we have by Lemma \ref{propsigmaepi} that 
\begin{equation} \label{eqq3} [\overline{\chi}] \notin \Sigma^1(G/ G '').\end{equation} 
Applying again  Lemma \ref{propsigmaepi} for the canonical epimorphism $G \to G / G''$ we obtain that $$[\chi] \notin \Sigma^1(G).$$

	\section{Proofs of Proposition B and Corollary C}
	
	\subsection{Proof of Proposition B}
	Let $X$ be the set of vertices of the graph $\Gamma$. Since the $\Sigma^1$-Conjecture for Artin groups holds for $G$ we have that  $[\chi] \notin \Sigma^1(G)$ if and only if  $\mathcal{L}(\chi)$ is either disconnected or not dominant in $\Gamma$. In the latter  case it means that for $X_1 = \{ x \in X ~| ~\chi(x) = 0 \}$ and $X_2 = \{ x \in X ~|  ~  \chi(x) \not= 0 \}$ there is a vertex $x \in X_1$ that is not connected by an edge with  any vertex from $X_2$. 
	 	 
	  Define
	$$
	K = \{ Y_1 ~| ~Y_1 \hbox{ is a subset of }X  \hbox{  such that there is an element of }Y_1 $$
	$$
	\hbox{  that is not linked by an edge with any element of } X \setminus Y_1 \}.
	$$ Define $$\Delta(Y_1) = \{ [\chi] \in S(G) ~| ~\chi(Y_1) = 0 \}.$$ Note that if $[\chi] \in \Delta(Y_1)$ for some $Y_1 \in K$ then $Y_1 \subseteq \{ v \in X ~| ~\chi(v) = 0 \}$ and there is an element of $Y_1$ that is not linked by an edge with any element of $\{ v \in X ~| ~\chi(v) \not= 0 \} \subseteq X \setminus Y_1$. Thus $\mathcal{L}(\chi)$ is not dominant in $\Gamma$ and hence
	$$\cup_{Y_1 \in K}  \Delta(Y_1) \subseteq \Sigma^1(G) ^c = S(G) \setminus \Sigma^1(G).$$

	Now we analyse the condition that $\mathcal{L}(\chi)$  is disconnected. 	 We define
	$$N = \{ (Y_1, e_1, \ldots, e_m) ~| ~ Y_1 \subset X, ~ e_1, \ldots, e_m \hbox{ are edges with even labels bigger than 2 in the} $$ $$\hbox{  complete subgraph } \mathcal{L}_0 \hbox{ of } \Gamma \hbox{  spanned by }X \setminus  Y_1 \hbox{ and }  \mathcal{L}_0 \setminus \{ e_1, \ldots, e_m\} \hbox{ is disconnected} \}.$$
	
	Define for $(Y_1, e_1, \ldots, e_m) \in N$
	$$\Delta_0(Y_1, e_1, \ldots, e_m) = \{ [\chi] \in S(G) ~| ~\chi(Y_1) = 0, \chi(y) \not= 0 \hbox{ for } y \in X \setminus Y_1,$$ $$ \chi ( u_i v_i) = 0, \hbox{ where } u_i, v_i \hbox{ are the vertices of } e_i, 1 \leq i \leq m \}$$ and 
	$$\Delta(Y_1, e_1, \ldots, e_m) = \{ [\chi] \in S(G) ~| ~\chi(Y_1) = 0, $$ $$ \chi ( u_i v_i) = 0,  \hbox{ where } u_i, v_i \hbox{ are the vertice of } e_i, 1 \leq i \leq m \}.$$
	Since the $\Sigma^1$-Conjecture for Artin groups holds for $G$
	$$\cup_{(Y_1, e_1, \ldots, e_m) \in N} \Delta_0(Y_1, e_1, \ldots, e_m) \subseteq \Sigma^1(G)^c.$$
	Since $\Sigma^1(G)^c$ is a closed subset of $S(G)$ we deduce that
	$$\cup_{(Y_1, e_1, \ldots, e_m) \in N} \Delta(Y_1, e_1, \ldots, e_m) \subseteq \Sigma^1(G)^c.$$
	Using again that the $\Sigma^1$-Conjecture for Artin groups holds for $G$ we have
	$$\Sigma^1(G)^c = (\cup_{(Y_1, e_1, \ldots, e_m) \in N} \Delta_0(Y_1, e_1, \ldots, e_m)) \cup (\cup_{Y_1 \in K}  \Delta(Y_1) ) \subseteq$$ $$ (\cup_{(Y_1, e_1, \ldots, e_m) \in N} \Delta(Y_1, e_1, \ldots, e_m)) \cup (\cup_{Y_1 \in K}  \Delta(Y_1) ) \subseteq \Sigma^1(G)^c.$$
	Thus 
	$$\Sigma^1(G)^c = (\cup_{(Y_1, e_1, \ldots, e_m) \in N} \Delta(Y_1, e_1, \ldots, e_m)) \cup (\cup_{Y_1 \in K}  \Delta(Y_1) ) $$
	and by construction each $\Delta(Y_1, e_1, \ldots, e_m)$ and 
	$\Delta(Y_1)$ is of the type $S(G, H)$ for appropriate subgroups $H$ of $G$.

	\subsection{Proof of Corollary C}  
	
	I. Let $[\chi] \in S(G)$ and $\overline{\chi}$ be the character of $G/ G''$ induced by $\chi$.
	By Lemma \ref{propsigmaepi} if $[\chi] \in \Sigma^1(G)$ then $[\overline{\chi}] \in  \Sigma^1(G/ G'')$. 
	
	\medskip
	II. Suppose now that $[\chi] \notin \Sigma^1(G)$. We aim to show that $[\overline{\chi}] \notin \Sigma^1(G/ G'')$. By Theorem A either $\mathcal{L}(\chi)$ is disconnected or $\mathcal{L}(\chi)$ is not dominant in $\Gamma$. We consider several cases.
	
1)	We suppose first that $\mathcal{L}(\chi)$ is not dominant in $\Gamma$. Then there is a vertex $t$ of $\Gamma$ such that $\chi(t) = 0$ and $t$ is not connected by an edge with any vertex of $\mathcal{L}(\chi)$ i.e. with any vertex $v$ in $\Gamma$ with $\chi(v) \not= 0$. Let $S$ be the graph obtained from $\Gamma$ by deleting all vertices $w$ of $\Gamma$ and the edges attached to them such that $w \not= t$ and $\chi(w) = 0$. Thus the set of vertices of $S$ is $V(\mathcal{L}(\chi)) \cup \{ t \}$. We write $V_0$ for $V(\mathcal{L}(\chi))$. Furthermore in $S$ all vertices from $V_0$ are linked by edges labelled by 2. We define $Q_0$ as the free abelian group with a basis $V_0$,  and consider the epimorphism  of groups $$\pi : G = G_{\Gamma} \to G_S = Q_0 * \langle t \rangle$$
	that is the identity on the vertices of $S$ and sends all vertices  $w$ to $1$, where $w \not= t$ and $\chi(w) = 0$. Here $\langle t \rangle$ is infinite cyclic. We set $$Q = G_S / G_S' \simeq Q_0 \oplus \mathbb{Z}.$$ Since $\Ker (\pi) \subseteq \Ker(\chi)$ we deduce that $\chi$ induces a character of $G_S$ that we denote by $\chi_S$. Finally it remains to show that $[\chi_S] \in \Sigma^1(G_S/ G_S'')^c$ and applying again Lemma \ref{propsigmaepi} for the epimorphism $G / G '' \to G_S/ G_S''$ induced by $\pi$ we can deduce that $[\chi] \in \Sigma^1(G / G '')^c$. 
	
	We will show that $\Sigma^1(G_S/ G_S'')^c = S(G_S/ G_S'')$.
		By \cite{B-S2} 
$[\mu ]\in \Sigma^1(G_S/ G_S'')$ if and only if for every prime ideal $P$ of $\mathbb{Z} Q$ minimal subject to $I \subseteq P$ we have that $\mathbb{Z} Q / P$ is finitely generated as $\mathbb{Z} Q_{\mu_0}$-module, where $\mu_0$ is the character of $G_S/ G_S'$ induces by $\mu$, $I = ann_{\mathbb{Z} Q} (G_S'/ G_S'')$ and we view $G_S'/ G_S''$ as $\mathbb{Z} Q$-module with $Q$ acting on the left by conjugation. Thus if we show that $I = 0$ we are done, since in this case $P = 0$ and $\mathbb{Z} Q / P \simeq \mathbb{Z} Q$ is not finitely generated as $\mathbb{Z} Q_{\mu_0}$-module for any non-zero real character $\mu_0$.

Note that $G_S'/ G_S''$ is generated as $\mathbb{Z} Q$-module by $w(q_i) : = t^{ -1} q_i^{ -1} t  q_i $ for $ 1 \leq i \leq m$, where $V_0 = \{ q_1, \ldots, q_m \}$ and has relations that come from the Hall-Witt identity (where all commutators are left normed)
$$1 = [q_i, t^{ - 1}, q_j]^t . [t, q_j^{ - 1}, q_i]^{ q_j} . [q_j, q_i^{ - 1}, t]^{ q_i} =[q_i, t^{ - 1}, q_j]^t . [t, q_j^{ - 1}, q_i]^{ q_j} $$ 
Since $[q_i, t^{ - 1}] = q_i^{-1} t q_i t^{-1} =  t w(q_i) t^{  -1} $   we have 
$$[q_i, t^{ - 1}, q_j]^t = ( t w(q_i)^{ -1}  t^{  -1}  q_j^{ -1} t w(q_i) t^{  -1} q_j)^t = w(q_i)^{ -1}  t^{  -1}  q_j^{ -1} t w(q_i) t^{  -1} q_j t=$$
$$w(q_i)^{ -1}  w(q_j) q_j^{ -1} w(q_i) q_j w(q_j)^{ -1}= w(q_i)^{ -1}  w(q_j) (w(q_i)^{  q_j}) w(q_j)^{ -1}.$$
Since  $[t, q_j^{  -1}] = t^{ -1} q_j t q_j^{ - 1} = q_j w(q_j) ^{ -1 } q_j^{ -1} $ we have 
$$[t, q_j^{ - 1}, q_i]^{ q_j}  =  ( q_j w(q_j) q_j^{ -1 } q_i^{-1} q_j w(q_j) ^{ -1 } q_j^{ -1} q_i)^{ q_j} =$$ $$ (w(q_j)^{  q_j^{ -1} }  (w(q_j)^{  q_j^{ -1 }q_i})^{ -1 })^{ q_j} = w(q_j) (w(q_j)^{ q_j^{ -1 } q_i q_j}  )^{ -1 }.$$
Then in $G_S/ G_S''$ since $[w(q_k), w(q_s)^{  q_p} ] = 1$ and $[q_k, q_s] = 1$ for $1 \leq k,s,p \leq m$   we have
$$1 = [q_i, t^{ - 1}, q_j]^t . [t, q_j^{ - 1}, q_i]^{ q_j}  = w(q_i)^{ -1}  w(q_j) (w(q_i)^{  q_j}) w(q_j)^{ -1}  w(q_j) (w(q_j)^{ q_j^{ -1 } q_i q_j}  )^{ -1 }=$$ \begin{equation} \label{hall} w(q_i)^{ q_j} .w(q_i)^{-1} w(q_j) (w(q_j)^{ q_i})^{ -1 }
= (^{ q_j^{ -1}} w(q_i) ).w(q_i)^{-1}. w(q_j) .(^{ q_i^{-1}} w(q_j))^{ -1 }.
\end{equation} 
Recall that we view $G_S'/ G_S''$ as a left $\mathbb{Z} Q$-module via conjugation and it is generated as $\mathbb{Z} Q$-module by $w(q_1), \ldots, w(q_m)$ subject only to the relations (\ref{hall}) since $G_S$ is the free product of $Q_0$ and $\langle t \rangle$. Then  (\ref{hall}) written  additively as an element of $\sum_i \mathbb{Z} Q w(q_i)$ becomes
$$0 = (q_j^{-1} -1) w(q_i) -  (q_i^{-1} -1) w(q_j).$$
Then in the localisation $\mathbb{Z} Q[ \frac{1}{q_1-1}, \ldots, \frac{1}{q_m-1}] \otimes_{\mathbb{Z} Q} (G_S'/ G_S'') $  we have (after omiting $\otimes$) $$w(q_i) = \frac{1}{q_j^{-1} -1 } (q_i^{-1} -1) w(q_j),$$ hence $\mathbb{Z} Q[ \frac{1}{q_1-1}, \ldots, \frac{1}{q_m-1}] \otimes_{\mathbb{Z} Q} (G_S'/ G_S'')$ is a cyclic $\mathbb{Z} Q[ \frac{1}{q_1-1}, \ldots, \frac{1}{q_m-1}] $-module and
$$\mathbb{Z} Q[ \frac{1}{q_1-1}, \ldots, \frac{1}{q_m-1}] \otimes_{\mathbb{Z} Q} (G_S'/ G_S'') \simeq \mathbb{Z} Q[ \frac{1}{q_1-1}, \ldots, \frac{1}{q_m-1}].$$
Finally since the ring homomorphism $\mathbb{Z} Q \to \mathbb{Z} Q[ \frac{1}{q_1-1}, \ldots, \frac{1}{q_m-1}]$ induced by the identity on $Q$ and $\mathbb{Z}$ is injective we deduce that 
 $I = 0$.
	
	2) Suppose now that  $\mathcal{L}(\chi)$ is disconnected and $\chi$ is a discrete character.
	Then in (\ref{eqq3}) from the proof of Theorem A we showed  that $[\overline{\chi}] \notin \Sigma^1(G/ G'')$.
	
	 3) Finally suppose that $\mathcal{L}(\chi)$ is disconnected without any restrictions on $\chi$.  By Theorem A and Proposition  B    $\Sigma^1(G)^c$  is a rationally defined polyhedron. In particular since $[\chi] \in \Sigma^1(G)^c$ there is a sequence of  discrete characters $\chi_i$ of $G$ such that $[\chi]$ is the limit of the sequence $[\chi_i]$ in $S(G)$ and $[\chi_i] \in \Sigma^1(G)^c$.   Then since the $\Sigma^1$- Conjecture for Artin groups holds for $G$ we have that either $\mathcal{L}(\chi_i)$ is disconnected  or $\mathcal{L}(\chi_i)$ is not dominant in $\Gamma$. Let $\overline{\chi}_i$ be the character of $G/ G''$ induced by $\chi_i$. Then by part 1) and 2) $[\overline{\chi}_i] \in \Sigma^1(G/ G'')^c$. Since $\Sigma^1(G/ G'')^c$ is a closed subset of $S(G/ G'')$ and  $[\overline{\chi}] $  is the limit of the sequence $ [\overline{\chi}_i] $ in $S(G/ G'')$ we deduce that $[\overline{\chi}] \in \Sigma^1(G/ G'')^c$.

	\section{Some examples} \label{final}
	
	Finally we remark why the method used in this paper generally does no apply for arbitrary Artin groups or even Artin groups that do not satisfy the assumptions of Theorem A. It boils down to the fact that in general the maximal metabelian quotient $G/ G ''$ does not contain sufficient information to calculate $\Sigma^1(G)$.
	
	\medskip
	
	1) Let $\Gamma$ be a graph with vertices $V \cup W$, where $V = \{ v, s \}$, $W = \{ u, w \}$, all vertices are linked by an edge  and $ 2m_{xy}$ denotes the label of the edge that links $x$ and $y$. We set
	$$m_{v,s} = m_{u,w} = 1, ~ m_{u,v} = m_{v,w} = m_{w,s} = 2, ~ m_{s,u} = 3.$$	
	
	Consider the character $\chi : G = G_{\Gamma} \to \mathbb{R} $ that sends $V$ to 1 and $W$ to -1.  
 Let $\overline{\chi}$ be the character of $G/ G ''$ induced by $\chi$ and $\chi_0$ be the character of $G/ G'$ induced by $\chi$.
	We will show  that $\Ker(\overline{\chi})$  is finitely generated though the $\Sigma^1$-Conjecture for Artin groups predicts that $[{\chi}], [-{\chi}] \notin \Sigma^1(G)$, hence if this prediction holds $\Ker(\chi)$ is not finitely generated but we do not know whether this is the case. Justifying this by concrete calculation is surprisingly hard.

 To calculate $G/ G''$ we use
the previous calculations from the proof of Theorem A  and we deduce that $G'/ G''$  is generated as $\mathbb{Z} Q$-module by  $e_v \wedge e_u, e_v \wedge e_w, e_s \wedge e_u, e_s \wedge e_w$ subject to the relations
$$ (v-1)e_s \wedge e_u  -(s-1) e_v \wedge e_u   = 0 = (v-1)e_s \wedge e_w  -(s-1) e_v \wedge e_w,$$  
$$(w-1)e_s \wedge e_u  -(u-1) e_s \wedge e_w   = 0  =(w-1)e_v \wedge e_u  -(u-1) e_v \wedge e_w $$ and $$
( 1 + uv) e_v \wedge e_u = 0 =( 1 + vw) e_v \wedge e_w, ( 1 + us + (us)^2) e_s \wedge e_u = 0 = ( 1 + sw)e_s \wedge e_w.$$
By \cite{B-S2} the $\Sigma^1(G/ G'')$ depends on the minimal prime ideals $P$ above $I = ann_{\mathbb{Z} Q} (A)$, where $A = G'/ G''$ and $Q = G/ G'$.   More precisely 
$[\overline{\chi} ]\in \Sigma^1(G/ G'')$ if and only if for every prime ideal $P$ of $\mathbb{Z} Q$ minimal subject to $I \subseteq P$ we have that $\mathbb{Z} Q / P$ is finitely generated as $\mathbb{Z} Q_{\chi_0}$-module.

If $I_1 = ann_{\mathbb{Z} Q} (e_v \wedge e_u), I_2 = ann_{\mathbb{Z} Q} (e_v \wedge e_w), I_3 = ann_{\mathbb{Z} Q} (e_s \wedge e_u), I_4= ann_{\mathbb{Z} Q} (e_s \wedge e_w) $ then
$$I_1 I_2 I_3 I_4 \subseteq I_1 \cap I_2 \cap I_3 \cap I_4 = I \subseteq P$$
hence for some $i$ we have $I_i \subseteq P$. Here we consider the case $i = 1$. The other cases are similar. 
Note that
$$ 1 + uv, ( 1 + su + (su)^2)(s-1), (1 + vw)(w-1),  ( 1 + sw)(s-1) (w - 1) \in I_1 \subseteq P.$$
If $s-1, w- 1 \notin P$, since $P$ is a prime ideal we get
$$ 1 + uv, 1 + su + (su)^2, 1 + vw,  1 + sw \in P.$$
But the ideal in $\mathbb{Z} Q$-generated by 	$1 + uv, 1 + su + (su)^2, 1 + vw,  1 + sw$ is the whole ring $\mathbb{Z} Q$, a contradiction.
Thus we can assume that $s-1$ or $w - 1$ is an element of $P$. Since $\chi(s) \not= 0, \chi(w) \not= 0$ we deduce that $\mathbb{Z} Q / P$ is finitely generated as $\mathbb{Z} Q_{\chi_0}$-module and as $\mathbb{Z} Q_{ - \chi_0}$-module. This holds for any proper prime ideal $P$ above $I_1$ and similarly for any proper prime ideal $P$ above $I_i$, hence for any proper prime ideal $P$ above $I$.  Then $[\overline{\chi} ], [- \overline{\chi}] \in \Sigma^1(G/ G'')$, so $\Ker (\overline{\chi})$ is finitely generated.

2) Consider  the graph $\Gamma$  that is triangle with labels 3,4 and 6. By \cite{A-K1} the $\Sigma^1$-Conjecture for Artin groups holds for $G = G_{\Gamma}$, hence $\Sigma^1(G) ^c \not= \emptyset$, hence $G'$ is not finitely generated. We can do calculations with Fox derivatives similar to the ones from Section \ref{Fox-section} and show that in this case $G'/ G''$ is a finitely generated abelian group  though $G'$ is not finitely generated.

\end{document}